\pdfoutput=1
\RequirePackage{ifpdf}
\ifpdf 
\documentclass[pdftex]{sigma}
\else
\documentclass{sigma}
\fi

\numberwithin{equation}{section}

\newtheorem{Teo}{Theorem}[section]
\newtheorem{Pro}[Teo]{Proposition}
\newtheorem{Claim}[Teo]{Claim}

{ \theoremstyle{definition}
\newtheorem{Def}[Teo]{Definition}
\newtheorem{Ob}[Teo]{Remark}}

\newcommand{\cA}{\mathcal{A}}
\newcommand{\tcA}{\tilde{\mathcal{A}}}
\newcommand{\bcA}{\bar{\mathcal{A}}}
\newcommand{\C}{\mathbb{C}}
\newcommand{\N}{\mathbb N}
\newcommand{\R}{\mathbb R}
\newcommand{\ep}{\varepsilon}
\newcommand{\be}{\beta}
\newcommand{\de}{\delta}
\newcommand{\De}{\Delta}
\newcommand{\lam}{\lambda}
\newcommand{\Lam}{\Lambda}
\newcommand{\diag}{\operatorname{diag}}

\begin{document}

\allowdisplaybreaks

\newcommand{\arXivNumber}{1510.05979}

\renewcommand{\PaperNumber}{104}

\FirstPageHeading

\ShortArticleName{Continuous Choreographies as Limiting Solutions of $N$-body Type Problems}

\ArticleName{Continuous Choreographies as Limiting Solutions\\ of $\boldsymbol{N}$-body Type Problems with Weak Interaction}

\Author{Reynaldo CASTANEIRA~$^\dag$, Pablo PADILLA~$^\dag$ and H\'ector S\'ANCHEZ-MORGADO~$^\ddag$}

\AuthorNameForHeading{R.~Castaneira, P.~Padilla and H.~S\'anchez-Morgado}

\Address{$^\dag$~Instituto de Investigaciones en Matem\'aticas Aplicadas y en Sistemas, UNAM,\\
\hphantom{$^\dag$}~M\'exico D.F.~04510, M\'exico}
\EmailD{\href{mailto:r.castaneira@matem.unam.mx}{r.castaneira@matem.unam.mx}, \href{mailto:pablo@mym.iimas.unam.mx}{pablo@mym.iimas.unam.mx}}
\Address{$^\ddag$~Instituto de Matem\'aticas, UNAM, M\'exico D.F.~04510, M\'exico}
\EmailD{\href{mailto:hector@matem.unam.mx}{hector@matem.unam.mx}}

\ArticleDates{Received October 20, 2015, in f\/inal form October 29, 2016; Published online October 31, 2016}

\Abstract{We consider the limit $N\to +\infty$ of $N$-body type problems with weak interaction, equal masses and $-\sigma$-homogeneous potential, $0<\sigma<1$. We obtain the integro-dif\/ferential equation that the motions must satisfy, with limit choreographic solutions corresponding to travelling
waves of this equation. Such equation is the Euler--Lagrange equation of a~corresponding limiting action functional. Our main result is that the circle is the absolute minimizer of the action functional among zero mean (travelling wave) loops of class~$H^1$.}


\Classification{70F45; 70G75; 70F10}

\section{Introduction}
It is well known that the planar regular $N$-gon relative equilibrium is a solution of the equations of motion for the Newtonian $N$-body problem with $N\geq3$ and equal mass bodies \cite{Barutello-Terracini, C-D, Perko-Walter}. We consider $N$-body type problems with weak interaction, equal masses in~$\mathbb{R}^d$, $d\geq2$, and $-\sigma$-homogeneous potential, $0 < \sigma< 1$, and study the resulting equation of motion when $N\to +\infty$, which we
will refer to as {\em continuous system}, as well as a particular type of solutions which we call {\em continuous choreographies} using a~variational approach. We can roughly say that a continuous choreography is a~limiting conf\/iguration of classical choreographies when the number of particles grows without limit and indeed the circle turns out to be the continuous choreography associated to regular $N$-gons on the plane. G.~Buck~\cite{Buck} considered curves which are locally approximated by solutions of the $N$-body problem. There is an essential dif\/f\/iculty in trying to study a continuous curve of particles
interacting under the Newtonian law, since the combined force exerted by the neighbours blows up, except at points where the curvature vanishes. It seems that overcome this dif\/f\/iculty is impossible. Therefore we conf\/ined ourselves to consider weak interaction.

The paper is structured as follows: In the introduction we recall the variational approach of the $N$-body problem, as well as the particular type of solutions that we consider in this paper, the choreographic solutions. In Section~\ref{sec:cont-chor-as} we obtain the limit equation of the $N$-body problem when $N\to +\infty$ and look at the particular case of travelling wave type continuous distributions of masses. Then we consider the limiting action functional and its critical points. In Section~\ref{section3} we f\/irst show that the circle is a continuous choreography for $0 <\sigma < 1$, and then prove our main result that it is in fact the absolute minimizer of the action functional among zero mean loops of class $H^1$ with period~1. This is the continuous distributions of masses counterpart of Theorem~1 in~\cite{Barutello-Terracini} according to which, under the choreographic restriction, the Lagrangian action attains its absolute minimum at the planar regular $N$-gon relative equilibrium.

\subsection[Choreographies for the $N$-body problem]{Choreographies for the $\boldsymbol{N}$-body problem}
The {\em $N$-body problem} with homogeneous potential of degree $\sigma\in(-1,0)$, is the study of the dynamics of N positive mass points $m_1, \dots,m_N$ moving in~$\R^d$, $d\ge 2$, interacting according to equations
\begin{gather}\label{Newton}
m_i\ddot{q}_{i}=-\underset{ j\neq i}{\sum_{ j=1,\dots,N}}\sigma m_im_j\frac{q_{i}-q_{j}}{\|q_{i}-q_{j}\|^{2+\sigma}}, \qquad
i=1,\dots,N,
\end{gather}
where $q_{i}(t) \in \R^d$ denotes the position of $m_i$ at time $t$. The case $\sigma=1$ is the {\em Newtonian $N$-body problem}.

Equivalently, we have the second-order equation
\begin{gather}\label{C}
M\ddot{q}=\nabla U(q), \qquad q=(q_{1}, \dots,q_{N}),
\end{gather}
where $U\colon \R^{d\times N}\to [0,\infty]$ given by
\begin{gather*}
U(q)=\sum_{1\leq i<j\leq N}\frac{m_i m_j}{\|q_{i}-q_{j}\|^\sigma},
\end{gather*}
is the {\em potential function} (the negative of the {\em potential energy}), $\nabla$ is the {\em gradient} in $\R^{d\times N}$, $M=\diag[m_1I_d, \dots,m_NI_d]$.

Since the {\em center of mass} describes an uniform motion, we can f\/ix it at the origin, and take
\begin{gather*}
\mathcal{X}=\left\{q\in \R^{d\times N}\Big|\sum_{i=1}^N m_iq_i =0\right\}
\end{gather*}
as conf\/iguration space for \eqref{C}.

According to the variational approach, solving \eqref{Newton} is equivalent to seeking the critical points of the Lagrangian action
\begin{gather}\label{Action}
A^{\sigma}\colon \ H^1([0,1],\mathcal{X}) \to[0,+\infty], \qquad A^{\sigma}(q)=\int_0^1 L^{\sigma}(q(t),\dot{q}(t))dt,
\end{gather}
where $L^\sigma\colon T\mathcal{X}\to[0,+\infty]$ is def\/ined on the {\em tangent bundle}
\begin{gather*}
T\mathcal{X}=\big\{(q,v)\,\big|\, q\in \mathcal{X},\; v \in \R^{d\times N}\big\}
\end{gather*}
as
\begin{gather*}
L^{\sigma}(q,v):=K(v)+U(q)=\frac{1}{2}\sum_{i=1}^N m_i \|v_i\|^2+ \sum_{1\leq i<j\leq N}\frac{m_i m_j}{\|q_{i}-q_{j}\|^{\sigma}}.
\end{gather*}
Equations \eqref{Newton} are the {\em Euler--Lagrange equations} associated to the functional~\eqref{Action}.

There is an important class of solutions of \eqref{Newton} known as {\em choreographies}. A~{\em choreography} of~\eqref{Newton} is a~periodic solution~$q(t)$ whose orbit is the union of closed curves, each of these is the trajectory of at least two bodies. If the solution consists of only one closed curve, then we call this solution a {\em simple choreography}.

\begin{Def}[see \cite{CGMS,Chenciner-Montgomery, Yu}]
We say that a solution of the $N$-body problem \eqref{Newton} is a simple choreography if it is $T$-periodic and all bodies move on the same curve, interchanging their mutual positions after a time f\/ixed, $\tau=T/N$, that is, there exists a function $x\colon \R\to \R^{d}$ such that:
\begin{gather*}
q_{i}(t)=x(t+(i-1) \tau), \qquad i=1,\dots,N, \qquad t \in \R.
\end{gather*}
\end{Def}

An example of a simple choreography is the relative equilibrium associated to Lagrange's equilateral triangle and this fact was extended in~\cite{Perko-Walter} to the case of $N$ equal masses. For details on choreographies we refer the reader to~\cite{Yu} and the references therein.

\section[Continuous choreographies as the limit $N\to\infty$ of $N$-body type problems]{Continuous choreographies as the limit $\boldsymbol{N\to\infty}$\\ of $\boldsymbol{N}$-body type problems}\label{sec:cont-chor-as}
\subsection{The equation of motion}

Consider the $N$-body problem \eqref{Newton} with equal masses ($m_1=\cdots=m_N=1/N$) on $\R^d$ and $0<\sigma<1$
 \begin{gather}\label{Nplano}
\ddot{q}_{i}(t)=-\underset{j\neq i}{\sum_{j=1,\dots,N}}\frac{\sigma}N
\frac{q_{i}(t)-q_{j}(t)}{\|q_{i}(t)-q_j(t)\|^{2+\sigma}}, \qquad i=1,\dots,N.
\end{gather}
Consider \looseness=-1 $q\colon [0,1]\times[a,b]\to\R^d$, periodic on the f\/irst variable and twice dif\/ferentiable on the se\-cond variable. Let $\De s=1/N$ and suppose that the position of the $i$-th mass at time~$t$ is given by
\begin{gather*}
q_i(t)=q((i-1)\De s,t),
\end{gather*}
then equation \eqref{Nplano} becomes
\begin{gather}\label{SD}
\ddot q_i(t)=-\underset{j\neq i}{\sum_{j=1,\dots,N}}\sigma \frac{q((i-1)\De s,t)-q((j-1)\De s,t)}{\|q((i-1)\De s,t)-q((j-1)\De s,t)\|^{2+\sigma}} \De s, \qquad i=1,\dots, N.
\end{gather}
So when $N\to +\infty$, $\De s\to 0$, the limit of the right-hand side of equation \eqref{SD} can be written as a Cauchy principal value
\begin{gather}\label{SC}
\frac{\partial^2 q}{\partial t^2}(s,t)=-
\lim_{\de\to 0}\int_{s+\de}^{1+s-\de}\sigma
\frac{q(s,t)-q(r,t)}{\|q(s,t)-q(r,t)\|^{2+\sigma}}dr.
\end{gather}
A natural way to make the positions $q_i(t)$ satisfy the coreography condition is to take $q(s,t)$ of travelling wave type, that is
\begin{gather}\label{Onda}
q(s,t)=y(s-vt),
\end{gather}
where $y$ is a 1-periodic function in $C^2(\mathbb{R},\mathbb{R}^d)$.

Using expression \eqref{Onda} in equation \eqref{SC}, we obtain that $y$ must satisfy
\begin{gather}\label{SC1}
v^2\ddot{y}(s)=-\lim_{\de\to 0}\int_{s+\de}^{1+s-\de}\sigma\frac{y(s)-y(r)}{\|y(s)-y(r)\|^{2+\sigma}}dr, \qquad s\in\R.
\end{gather}
We denote by $C_1^2(\R,\R^d)$ the set of 1-periodic functions in $C^2(\R,\R^d)$ and by $H_1^1(\R,\R^d)$ the set of 1-periodic functions whose restriction to $[0,1]$ is in $H^1([0,1],\R^d)$.
\begin{Claim}
Assume $y\in C_1^2(\R,\R^d)$ defines a regular simple closed curve. Then the right-hand side of \eqref{SC1} is well defined.
\end{Claim}

\begin{proof}
There is a continuous function $g\colon \R\times[-\ep,\ep]\to\R^d$ such that
\begin{gather*}y(s+t)=y(s)+\dot{y}(s)t+g(s,t)t^2. \end{gather*}
Since $y$ is regular, for $s$ f\/ixed, the function
\begin{gather*}F(t)=\frac{(y(s+t)-y(s)) |t|^{2+\sigma}}{t\| y(s+t)-y(s)\|^{2+\sigma}}=
\frac{\dot{y}(s)+tg(s,t)}{\|\dot{y}(s)+tg(s,t)\|^{2+\sigma}}\end{gather*}
can be considered as a continuous function on $[-\ep,\ep]$ with $F(0)=\dot{y}(s)\|\dot{y}(s)\|^{-2-\sigma}$ and dif\/fe\-ren\-tiable at~$0$. Thus
\begin{gather*}G(t)=\left(\frac{y(s+t)-y(s)}{\| y(s+t)-y(s)\|^{2+\sigma}}
-\frac{\dot{y}(s)t}{\| \dot{y}(s)t\|^{2+\sigma}}\right)|t|^\sigma =\frac{F(t)-F(0)}{t}\end{gather*}
is bounded on $[-\ep,\ep]-\{0\}$ and then
\begin{gather*}\lim_{\de\to 0}\int_{-\ep}^{-\de}\frac{G(t)dt}{|t|^\sigma}+\int_{\de}^{\ep}\frac{G(t)dt}{|t|^\sigma}\end{gather*}
exists. Since $t\mapsto t|t|^{-2-\sigma}$ is an odd function we have
\begin{gather*}\int_{-\ep}^{-\de}\frac{tdt}{|t|^{2+\sigma}}+\int_{\de}^{\ep}\frac{tdt}{|t|^{2+\sigma}}=0,\end{gather*}
and then
 \begin{gather*}\lim_{\de\to 0}\left(\int_{s-\ep}^{s-\de}+\int_{s+\de}^{s+\ep}\right)\frac{(y(r)-y(s))dr}{\| y(s)-y(r)\|^{2+\sigma}}
=\lim_{\de\to 0}\int_{-\ep}^{-\de}\frac{G(t)dt}{|t|^\sigma}+\int_{\de}^{\ep}\frac{G(t)dt}{|t|^\sigma}.\tag*{\qed}
\end{gather*}
\renewcommand{\qed}{}
\end{proof}

\subsection{Variational approach to continuous choreographies (\ref{SC1})}
Consider the action functional
\begin{gather*}\cA^\sigma\colon \ \Lam\to[0,+\infty]\end{gather*}
given by
\begin{gather}\label{FC}
\cA^\sigma(y)=\int_0^1\frac{v^2}{2}\|\dot{y}(s)\|^2 ds +\frac{1}{2}\int_0^1\int_0^1 \frac{dr ds}{\|y(s)-y(r)\|^\sigma},
\end{gather}
where
\begin{gather*}\Lam:=\left\{y\in H_1^1\big(\R,\R^d\big) \Big |\int_0^1 y(s)ds=0\right\}.\end{gather*}
For $y\in\Lam$, we have that $\|y\|_{L^2[0,1]}\le\|\dot{y}\|_{L^2[0,1]}$ and therefore $\|y\|_{H^1[0,1]}$ is equivalent to $\|y\|_{L^2[0,1]}$.
\begin{Pro}
Assume $y\in C_1^2(\R,\R^d)$ defines a regular simple closed curve. Then $y$ is an extremal of the functional~\eqref{FC} if and only if it satisfies~\eqref{SC1}.
\end{Pro}
\begin{proof}
 For $y\in C_1^2(\R,\R^d)$ and any $z\in C_1^2(\R,\R^d)$ we have
\begin{gather}
\frac{d\cA^\sigma(y+\ep z)}{d\ep}\Big|_{\ep=0}
= \frac{d\ }{d\ep}\left\{ \int_0^1\frac{v^2}{2}\|\dot{y}(s)+\ep\dot{z}(s)\|^2 ds\right\}_{\ep =0} \nonumber\\
\hphantom{\frac{d\cA^\sigma(y+\ep z)}{d\ep}\Big|_{\ep=0}=}{}
 +\frac{1}{2}\frac{d\ }{d\ep}\left\{\int_0^1\int_0^1 \frac{dr ds}{\|y(s)-y(r)+\ep(z(s)-z(r))\|^\sigma}\right\}_{\ep=0}.\label{PV2}
\end{gather}
Dif\/ferentiating and then integrating by parts the f\/irst term on the right-hand side of \eqref{PV2} we obtain
\begin{gather*}
\frac{d\ }{d\ep}\left\{\int_0^1\frac{v^2}{2}\|\dot{y}(s)+\ep\dot{z}(s)\|^2 ds\right\}_{\ep =0} =\int_0^1-v^2\ddot{y}(s)\cdot z(s)ds.
\end{gather*}
We now consider the second term on the right-hand side of~\eqref{PV2}. For $\de>0$ small
\begin{gather*}
 \int_0^1\int_0^1 \frac{dr ds}{\|y(s)-y(r)+\ep(z(s)-z(r))\|^\sigma}\\
\qquad {} =\int_0^1\int_{s-\de}^{1+ s-\de} \frac{dr ds}{\|y(s)-y(r)+\ep(z(s)-z(r))\|^\sigma}\\
\qquad {} =\int_0^1\left(\int_{s+\de}^{1+s-\de}+\int_{s-\de}^{s+\de}\right) \frac{dr ds}{\|y(s)-y(r) +\ep(z(s)-z(r))\|^\sigma}.
\end{gather*}
We have
\begin{gather*}
 \frac 12\frac{d\ }{d\ep} \left\{\int_0^1\int_{s+\de}^{1+s-\de}\frac{dr ds}{\|y(r)-y(s)+\ep(z(r)-z(s))\|^\sigma}\right\}_{\ep=0}\\
\qquad{} =-\frac 12 \int_0^1\int_{s+\de}^{1+s-\de}\frac{\sigma(y(s)-y(r))\cdot(z(s)-z(r))}{\|y(s)-y(r)\|^{2+\sigma}}drds\\
\qquad {} =-\frac 12\int_0^{1-\de}\int_{s+\de}^{1-\max(0,\de-s)}\frac{\sigma(y(s)-y(r))\cdot(z(s)-z(r))}{\|y(s)-y(r)\|^{2+\sigma}}drds\\
\qquad\quad{} -\frac 12\int_0^{1-\de}\int_{r+\de}^{1-\max(0,\de-r)}\frac{\sigma(y(s)-y(r))\cdot(z(s)-z(r))}{\|y(s)-y(r)\|^{2+\sigma}}dsdr\\
\qquad{} =-\int_0^1\int_{s+\de}^{1+s-\de}\frac{\sigma(y(s)-y(r)) \cdot z(s)}{\|y(s)-y(r)\|^{2+\sigma}}drds.
\end{gather*}
There are continuous functions $g,h\colon \R\times[-\ep,\ep]\to\R^d$ such that
\begin{gather*}y(t+s)=y(s)+\dot{y}(s)t+g(s,t)t^2,\qquad z(t+s)=z(s)+\dot{z}(s)t+h(s,t)t^2.\end{gather*}
Since $y$ is regular, $|\dot{y}(s)|\ge\be>0$, for $\de, \ep>0$ small we have that
\begin{gather*}|t|\le\de, \quad |u|\le\ep \ \Rightarrow \ |\dot{y}(s)+u\dot{z}(s)+(g(s,t)+uh(s,t))t|\ge\frac\be2.\end{gather*}
The fundamental theorem of calculus gives
\begin{gather*}
\left|\int_{s-\de}^{s+\de}\frac{dr}{\|y(s)-y(r)+\ep(z(s)-z(r))\|^\sigma}-\frac{dr}{\|y(s)-y(r)\|^\sigma}\right|\\
\qquad{} =\left|\int_{s-\de}^{s+\de}\int_0^\ep \sigma\frac{(y(s)-y(r)+u(z(s)-z(r)))\cdot(z(s)-z(r))}
{\|y(s)-y(r)+u(z(s)-z(r))\|^{2+\sigma}}du dr\right|\\
\qquad{}=\left|\int_{-\de}^\de\int_0^\ep \sigma\frac{ (\dot{y}(s)+u\dot{z}(s)+(g(s,t)+uh(s,t))t)\cdot(\dot{z}(s)+h(s,t)t)}
{\|\dot{y}(s) +u \dot{z}(s)+(g(s,t)+u h(s,t))t\|^{2+\sigma}|t|^\sigma}du dt\right|\\
\qquad{} \le C\int_{-\de}^\de\int_0^\ep|t|^{-\sigma} dudt=\frac{2C}{1-\sigma}\ep\de^{1-\sigma}
\end{gather*}
with $C$ constant. Thus, for $\de>0$ small we have
\begin{gather*}
 \frac{d\cA^\sigma(y+\ep z)}{d\ep}\Big|_{\ep=0} = \int_0^1-v^2\ddot{y}(s) \cdot z(s)ds+O\big(\de^{1-\sigma}\big)\\
\hphantom{\frac{d\cA^\sigma(y+\ep z)}{d\ep}\Big|_{\ep=0} =}{}  -\int_0^1\int_{s+\de}^{1+s-\de}\frac{ \sigma(y(s)-y(r)) \cdot z(s)}{\|y(s)-y(r)\|^{2+\sigma}}drds,
\end{gather*}
and then
\begin{gather*}
\frac{d\cA^\sigma(y+\ep z)}{d\ep}\Big|_{\ep=0}=-\int_0^1\left[v^2\ddot{y}(s)+\lim_{\de\to 0}
\int_{s+\de}^{1+s-\de}\frac{\sigma(y(s)-y(r))}{\|y(s)-y(r)\|^{2+\sigma}}dr\right]\cdot z(s)ds.
\end{gather*}
Therefore, $y$ satisf\/ies the condition
\begin{gather*}
\frac{d\cA^\sigma(y+\ep z)}{d\ep}\Big|_{\ep=0}=0
\end{gather*}
for any $z\in C_1^2(\R,\R^d)$, if and only if it satisf\/ies \eqref{SC1}.
\end{proof}

\section{Circular choreography as minimizer of the action}\label{section3}

In this section we consider circular planar curves as solutions of equation \eqref{SC1} and in fact as absolute minimizers of the action functional. This is motivated by papers \cite{Perko-Walter, Xie-Zhang}, which prove that the planar regular $N$-gon relative equilibrium is a solution of the $N$-body
problem, and mainly by Theorem~1 in~\cite{Barutello-Terracini}. We naturally follow the ideas from that paper.
\begin{Pro} The function given by $x(s):=e^{2\pi i s}$ is a solution of \eqref{SC1} on the plane $\R^2\cong\C$ if and only if
\begin{gather}\label{circle}
v^2=\frac{\sigma}{4\pi^2} \lim_{\de\to 0}\int_\de^{1-\de}\frac{1-e^{2\pi it}}{|1-e^{2\pi it}|^{2+\sigma}}dt
=\frac{\sigma}{8\pi^2}\int_0^1\frac{dt}{(2\sin(\pi t))^\sigma}.
\end{gather}
\end{Pro}
\begin{proof}
We have
\begin{gather*}
\int_{s+\de}^{1+s-\de}\frac{e^{2\pi i s}-e^{2\pi i r}}{|e^{2\pi i s}-e^{2\pi i r}|^{2+\sigma}}dr
 =e^{2\pi i s}\int_{s+\de}^{1+s-\de}\frac{1-e^{2\pi i(r-s)}}{|1-e^{2\pi i (r-s)}|^{2+\sigma}}dr\\
\hphantom{\int_{s+\de}^{1+s-\de}\frac{e^{2\pi i s}-e^{2\pi i r}}{|e^{2\pi i s}-e^{2\pi i r}|^{2+\sigma}}dr}{}
=e^{2\pi i s}\int_\de^{1-\de}\frac{1-e^{2\pi i t}}{|1-e^{2\pi i t}|^{2+\sigma}}dt,
\end{gather*}
and
\begin{gather*}\ddot{x}(s)=-4\pi^2 e^{2\pi i s}.\end{gather*}
Thus, \eqref{SC1} is equivalent to \eqref{circle}
\end{proof}

\begin{Pro}\label{holder}
For $\be>0$ let $\mu\in C((0,1),\R^+)\cap L^{\be/\be+1}(0,1)$. For any $\xi\colon (0,1)\to\R^+$ we have
\begin{gather*}\left(\int_0^1\mu^{\be/\be+1}\right)^{\be+1}\le \left(\int_0^1\mu \xi\right)^\be\int_0^1\xi^{-\be}, \end{gather*}
and equality holds if an only if $\mu\xi^{\be+1}$ is constant.
\end{Pro}
\begin{proof}
By H\"older's inequality
\begin{gather*}\int_0^1\mu^{\be/\be+1}=\int_0^1(\mu\xi)^{\be/\be+1}\xi^{-\be/\be+1}
\le \left(\int_0^1\mu \xi\right)^{\be/\be+1}\left(\int_0^1\xi^{-\be}\right)^{1/\be+1},\end{gather*}
and equality holds if an only if $\mu\xi^{\be+1}$ is constant.
\end{proof}

Henceforth we let $v$ be given by \eqref{circle}. Def\/ine
\begin{gather*}
 \xi_y(t) =\int_0^1\|y(s+t)-y(s)\|^2ds, \qquad \hat\xi(t)=|1-e^{2\pi i t}|^2,\\
c =\int_0^1\hat\xi^{-\sigma/2}dt=\int_0^1\frac{dt}{(2\sin(\pi t))^\sigma}=\frac{8\pi^2v^2}\sigma,
\qquad \mu=\frac{\hat\xi^{-\sigma/2-1}}c.
\end{gather*}
If $y$ is twice dif\/ferentiable, then the integral
\begin{gather*}\De^\mu y(t)=\int_0^1\mu(s)(2y(t)-y(t+s)-y(t-s))ds
=\int_0^1\frac{2y(t)-y(t+s)-y(t-s)}{c(2\sin(\pi s))^{2+\sigma}}ds\end{gather*}
converges since the numerator is $\mathcal O((s-b)^2)$ as $s\to b$ for $b=0,1$.

Note that $\De^\mu$ is linear and for $f_k(t)=a e^{2k\pi it}$ we have
\begin{gather*}\De^\mu f_k(t)=
a e^{2k\pi it}\int_0^1\frac{2-e^{2k\pi is}-e^{-2k\pi is}}{c(2\sin(\pi s))^{2+\sigma}}ds
=d_kf_k(t),\end{gather*}
where the integral
\begin{gather*}d_k=\int_0^1\frac{4\sin^2(k\pi s)}{c(2\sin(\pi s))^{2+\sigma}}ds\end{gather*}
converges since the integrand is $\mathcal O(|s-b|^{-\sigma})$ as $s\to b$ for $b=0,1$.

By Proposition~2 in~\cite{Barutello-Terracini}, for $k\ge 2$, $u\in (0,\pi)$ we have $\sin^2(ku)<k^2\sin^2u$ and so $ d_k<k^2d_1$.
\begin{Ob}
We write $a\in\C^d$ as $a=\operatorname{Re} a+ i\operatorname{Im} a$ and $\bar a=\operatorname{Re} a- i\operatorname{Im} a$.
\end{Ob}
\begin{Pro}\label{estima-cinetica}
For every $y\in \Lam$
\begin{gather}\label{EC}
\int_0^1\|\dot{y}\|^2 \ge 4\pi^2 \int_0^1\mu\xi_y,
\end{gather}
and the equality holds if only if $y(t)=a e^{2\pi i t}+\bar a e^{-2\pi i t}$.
\end{Pro}
\begin{proof}
\begin{align*}
\int_0^1 \De^\mu y \cdot y
&=\int_0^1\mu(s)\int_0^1 (2\|y(t)\|^2- y(t)\cdot y(t+s)-y(t)\cdot y(t-s))dt ds\\
&= \int_0^1\mu(s)\int_0^1 (2\|y(t)\|^2-2y(t)\cdot y(t+s))dt ds\\
&= \int_0^1\mu(s)\int_0^1\|y(t)-y(t+s)\|^2dt ds= \int_0^1 \mu\xi_y.
\end{align*}
Observe that the inequality \eqref{EC} holds true for $y$ constant. We can then assume that $\xi_y$ never vanishes. Our aim is to show that the functional
\begin{gather*}
J(y):=\frac{\int_0^1 \|\dot{y}\|^2}{\int_0^1\De^\mu y \cdot y}
\end{gather*}
def\/ined on $\Lam$ attains its inf\/imum and that its minimal value is $4\pi^2$. Since the functional $J$ is homogeneous of degree zero, minimizing $J$ is equivalent to minimizing the coercive functional $\tilde J(y)=\int_0^1\|\dot{y}(t)\|^2dt$ on the constraint
\begin{gather*}
M=\left\{y \in \Lam\,\big| \int_0^1 \De^\mu y \cdot y dt =1\right\},
\end{gather*}
which is closed with respect to the weak $H_1^1$-topology. Let $(y_n)_n$ be a minimizing sequence, then $\int_0^1\|\dot{y}_n\|^2dt=\|\dot{y}_n\|_{L^2}$ is bounded. By lower semi-continuity of the norm we can select a~subsequence $(y_{n_k})$, weakly convergent to $y \in M$ and obtain
\begin{gather*}
\tilde{J}(y)\leq\liminf_{n_k}\tilde{J}(y_{n_k}).
\end{gather*}
So we can state that the minimum of $\tilde J$ exist. The minimal value of $\tilde J$ corresponds to the f\/irst eigenvalue $\lam_{\min}$ for the problem
\begin{gather}\label{problem-constraint}
 -\ddot{y}=\lambda \De^\mu y, \qquad  y\in H_1^1\big(\R,\R^d\big), \qquad \int_0^1 y=0.
\end{gather}
To study problem \eqref{problem-constraint} write a solution as
\begin{gather*}y(t)=\sum_{k=-\infty}^\infty a_k e^{2k\pi i t},\qquad a_0=0, \qquad a_{-k}=\bar a_k\end{gather*}
to get
\begin{gather*}\ddot y(t)=-\sum_{k=-\infty}^\infty 4\pi k^2a_k e^{2k\pi i t}, \qquad \De^\mu y(t)=\sum_{k=-\infty}^\infty a_k d_ke^{2k\pi i t},\end{gather*}
so that the eigenvalues of the problem~\eqref{problem-constraint} are $\lam_k=4\pi^2k^2/d_k$, $k\in\N$. Thus $\lam_{\min}=\lam_1=4\pi^2$ and the minimum is attained if and only if $y(t)=a e^{2\pi i t}+\bar a e^{-2\pi i t}$.
\end{proof}

\begin{Teo}\label{circle-min}
The absolute minimum of $\cA^\sigma$ on $\Lam$, with $0<\sigma<1$ and $v$ given by \eqref{circle}, is attained at and only at a unit circle
\begin{gather}\label{eq:1}
 y(t)=ae^{2\pi i t}+\bar a e^{-2\pi i t}=E_1\cos(2\pi t)+ E_2 \sin(2\pi t),
\end{gather}
where $E_1=2\operatorname{Re}(a)$, $E_2=2\operatorname{Im}(a)$ are orthogonal unit vectors in $\R^d$.
\end{Teo}
\begin{proof} By Jennsen's inequality
 \begin{gather*}\xi_y(t)^{-\sigma/2}=\left(\int_0^1\|y(s+t)-y(s)\|^2ds\right)^{-\sigma/2}\le \int_0^1\frac{ds}{\|y(s+t)-y(s)\|^\sigma},\end{gather*}
and equality holds if and only if $\xi_y(t) =\|y(t)-y(0)\|=\|y(s+t)-y(s)\|$ for any~$s$. Thus
\begin{gather}\label{cAjen}
\cA^\sigma(y)\ge \int_0^1 \frac {v^2}2\|\dot{y}\|^2 +\frac 12\int_0^1\xi_y^{-\sigma/2}.
\end{gather}
Def\/ine
\begin{gather*}
\tcA^\sigma(y):=\int_0^1 \frac {v^2}2\|\dot{y}\|^2 +\frac{4\pi^2v^2}\sigma\left(\int_0^1\mu\xi_y\right)^{-\sigma/2}.
\end{gather*}
By Proposition \ref{holder}, the minimum of the functional
\begin{gather*}\Phi(\xi)=\left(\int_0^1\mu \xi\right)^{\sigma/2}\int_0^1\xi^{-\sigma/2}\end{gather*}
is attained at $\xi$ if and only if $\xi$ is proportional to $\hat\xi$ and its value is $c=8\pi^2v^2/\sigma$, which together with~\eqref{cAjen} gives
\begin{gather}\label{cAtcA}
\cA^\sigma(y) \ge\tcA^\sigma(y),
\end{gather}
and equality in \eqref{cAtcA} holds at a circle{\samepage
\begin{gather*}y(t)=ae^{2\pi i t}+\bar a e^{-2\pi i t}=E_1\cos(2\pi t)+ E_2 \sin(2\pi t),\end{gather*}
where $E_1=2\operatorname{Re}(a)$, $E_2=2\operatorname{Im}(a)$ are orthogonal vectors in $\R^d$ of the same length.}

Def\/ine
\begin{gather*}\bcA^\sigma(y):=2\pi^2v^2 \left(\int_0^1\mu\xi_y +\frac 2\sigma\left(\int_0^1\mu\xi_y \right)^{-\sigma/2}\right).\end{gather*}
By Proposition \ref{estima-cinetica} we have
\begin{gather}\label{tcAbcA}
\tcA^\sigma(y) \ge\bcA^\sigma(y),
\end{gather}
and the equality in \eqref{tcAbcA} holds if only if $y(t)=a e^{2\pi i t}+\bar a e^{-2\pi i t}$.

The function $g\colon \R^+\to\R$, $g(u)=u+\frac 2\sigma u ^{-\sigma/2}$ has a unique minimum at $u=1$, therefore~$\bcA^\sigma$ attains its absolute minimum $2\pi^2v^2(1+2/\sigma)$ at functions $y\in\Lam$ with $\int_0^1\mu\xi_y=1$, among others at a unit circle~\eqref{eq:1}, and at a unit circle all $\cA^\sigma$, $\tcA^\sigma$, $\bcA^\sigma$ coincide. Conversely, if~$\cA^\sigma$ attains its minimum at $y\in\Lam$ then $\int_0^1\mu\xi_y=1$ and \eqref{cAtcA},
\eqref{tcAbcA} must be equalities. Thus $\xi_y(t) =\|y(t)-y(0)\|=\hat\xi(t)$ and $y(t)=a e^{2\pi it}+\bar a e^{-2\pi it}$. It is not dif\/f\/icult to see that $E_1=2\operatorname{Re}(a)$, $E_2=2\operatorname{Im}(a)$ are orthogonal unit vectors in~$\R^d$.
 \end{proof}

\section{Open problems}
It is possible that imposing additional symmetries beyond the choreography condition, could give some other critical points of the action besides the absolute minimizers. One also could impose some topological non-trivial condition to the curve def\/ining the continuous coreography. Connections with other topics such as vortex lines in f\/luid dynamics might also be of interest.

\subsection*{Acknowledgement}
We would like to thank R.~Montgomery and C.~Garc\'ia Azpeitia for pointing out mistakes in earlier versions of the paper. R.~Castaneira is grateful to R.~Montogomery for all his support to visit him at UC Santa Cruz. The authors thank the referees for useful observations.

\pdfbookmark[1]{References}{ref}
\LastPageEnding

\end{document}